\documentclass[preprint,review]{elsarticle}

\usepackage{amsthm}
\usepackage[utf8]{inputenc}
\usepackage{amsmath}
\usepackage{amsfonts}

\usepackage{amssymb}
\usepackage{enumerate}
\usepackage{subcaption}
\usepackage{algorithm}
\usepackage{algpseudocode}
\usepackage{tikz}
\usepackage{centernot}
\usepackage{lineno}

\usepackage{array}
\newcolumntype{C}{ >{\centering\arraybackslash} m }

\usepackage[colorlinks=true,breaklinks=true,bookmarks=true,urlcolor=blue,
     citecolor=blue,linkcolor=blue,bookmarksopen=false,draft=false]{hyperref}

\usepackage{cleveref}

\newcommand{\R}{\mathbb R}

\newcommand\veb{{\ve b}}
\newcommand\vecc{{\ve c}}
\newcommand\ved{{\ve d}}

\newcommand\veg{{\ve g}}

\newcommand\veu{{\ve u}}

\newcommand\vex{{\ve x}}
\newcommand\vey{{\ve y}}
\newcommand\vez{{\ve z}}

\usepackage{bm}
\newcommand{\vedelta}{\bm{\delta}}

\theoremstyle{definition}
\newtheorem{definition}{Definition}

\theoremstyle{remark}
\newtheorem{remark}{Remark}

\theoremstyle{plain}
\newtheorem{theorem}{Theorem}
\newtheorem{proposition}{Proposition}
\newtheorem{lemma}{Lemma}

\DeclareMathOperator{\rank}{rank}

\DeclareMathOperator{\poly}{poly}

\newcommand{\PP}{{\sf P}\xspace}
\newcommand{\NP}{{\sf NP}\xspace}
\newcommand{\NPh}{\hbox{{\sf NP}-hard}\xspace}

\newcommand{\CC}{\mathcal{C}}
\newcommand{\Oh}{\mathcal{O}}

\def\ve#1{\mathchoice{\mbox{\boldmath$\displaystyle\bf#1$}}
{\mbox{\boldmath$\textstyle\bf#1$}}
{\mbox{\boldmath$\scriptstyle\bf#1$}}
{\mbox{\boldmath$\scriptscriptstyle\bf#1$}}}

\theoremstyle{plain}
\makeatletter
\newtheorem*{rep@theorem}{\rep@title}
\newcommand{\newreptheorem}[2]{%
	\newenvironment{rep#1}[1]{%
		\def\rep@title{#2 \ref{##1}}%
		\begin{rep@theorem}}%
		{\end{rep@theorem}}}
\makeatother

\usepackage{framed}
\usepackage{tabularx}
\usepackage{tikz}
\usetikzlibrary{calc}

\newlength{\RoundedBoxWidth}
\newsavebox{\GrayRoundedBox}
\newenvironment{GrayBox}[1]%
{\setlength{\RoundedBoxWidth}{\columnwidth}
\addtolength{\RoundedBoxWidth}{-3ex}
\addtolength{\RoundedBoxWidth}{-.5pt}
	\def\boxheading{#1}
	\begin{lrbox}{\GrayRoundedBox}
		\begin{minipage}{\RoundedBoxWidth}}%
		{   \end{minipage}
	\end{lrbox}
	\begin{center}
		\begin{tikzpicture}%
		\node(Text)[draw=black,fill=white,inner sep=1.5ex,text width=\RoundedBoxWidth]
		{\usebox{\GrayRoundedBox}};
		\coordinate(x) at (current bounding box.north west);
		\node [draw=white,rectangle,inner sep=3pt,anchor=north west,fill=white]
		at ($(x)+(6pt,.75em)$) {\boxheading};
		\end{tikzpicture}
\end{center}}

\newenvironment{defproblemx}[1]{\noindent\ignorespaces%
	\FrameSep=6pt%
	\parindent=0pt%
	                 \vspace*{-1em}
	\begin{GrayBox}{\textsc{#1}}%
		\vspace*{.2em}
		\begin{tabular*}{\RoundedBoxWidth}{@{\extracolsep{\fill}} >{\itshape} p{1.9cm} p{0.75\columnwidth} @{\hfill}}%
		}{
		\end{tabular*}%
	\end{GrayBox}%
	\ignorespacesafterend
}

\newcommand{\defProblemFind}[3]{%
	\begin{defproblemx}{#1}
		Input: & #2 \\
		Find: & #3
	\end{defproblemx}
}

\newcommand{\defProblemDecide}[3]{%
	\begin{defproblemx}{#1}
		Input: & #2 \\
		Decide: & #3
	\end{defproblemx}
}

\usepackage{xspace}
\usepackage[textwidth=20mm,obeyFinal,colorinlistoftodos]{todonotes}
\newcommand{\mytodo}[2]{\todo[size=\tiny, color=#1!50!white]{#2}\xspace}

\newcommand{\mkcom}[1]{\mytodo{orange}{#1}}

\newreptheorem{theorem}{Theorem}
\newreptheorem{lemma}{Lemma}

\begin{document}
\journal{Operations Research Letters}

\begin{frontmatter}
\title{A Note on the Approximability of \\Deepest-Descent Circuit Steps}

\author[1]{Steffen Borgwardt}
\ead{steffen.borgwardt@ucdenver.edu}

\author[2]{Cornelius Brand\corref{cor1}}
\cortext[cor1]{Corresponding author}
\ead{cbrand@iuuk.mff.cuni.cz}
		
\author[2]{Andreas Emil Feldmann}
\ead{feldmann.a.e@gmail.com}

\author[2]{Martin Koutecký}
\ead{koutecky@iuuk.mff.cuni.cz}

\address[1]{Department of Mathematical and Statistical Sciences,
University of Colorado Denver \\
P.O. Box 173364, CB 170 \\
Denver, CO 80217-3364 USA}

\address[2]{Univerzita Karlova,
Matematicko-fyzikální fakulta,
Informatický ústav Univerzity Karlovy \\
Malostranské nám. 25 \\
118 00 Praha 1, Czech Republic}

\begin{abstract}
Linear programs (LPs) can be solved by polynomially many moves along the circuit direction improving the objective the most, so-called deepest-descent steps (dd-steps).
Computing these steps is NP-hard~(De Loera et al., arXiv, 2019), a consequence of the hardness of deciding the existence of an optimal circuit-neighbor (OCNP) on LPs with non-unique optima.

We prove OCNP is easy under the promise of unique optima,
but already $O(n^{1-\varepsilon})$-approximating dd-steps remains hard even for totally unimodular $n$-dimensional 0/1-LPs with a unique optimum.
We provide a matching $n$-approxi\-ma\-tion.

%
%
\end{abstract}

\end{frontmatter}

\section{Introduction}
Linear programming is a fundamental tool in both the theory and applications of combinatorial optimization: We are given a system $A\vex = \veb, \, B\vex \leq \ved$ with $A \in \R^{m_A \times n}, B \in \R^{m_B \times n}, \veb \in \R^{m_A}$ and $\ved \in \R^{m_B}$ and a cost vector $\vecc \in \R^n$.
We call an assignment $\vex \in \R^n$ to the variables \emph{feasible} if it satisfies the system of equalities and inequalities, and the set of these feasible assignments is a \emph{polyhedron}, which will be denoted as $P$ throughout. The goal is to find a feasible assignment $\vex \in \R^n$ minimizing $\vecc^T \vex$.

Linear programming has been known to be solvable in weakly-polynomial\footnote{From here on out, whenever we speak of a problem with instances containing numbers as inputs as being solvable in polynomial time, we intend this to mean weakly-polynomial time, unless explicitly stated otherwise.} time since the groundbreaking work of Khachiyan on the ellipsoid method \cite{khachiyan1979polynomial} and Karmarkar on the interior point method \cite{karmarkar1984new}. The existence of a strongly polynomial algorithm for linear programming, that is, an algorithm which makes $\poly(n, m_A, m_B)$ arithmetic operations and finds an optimal solution, is a major open problem.
Exploring methods other than the ellipsoid and interior point methods is a possible pathway for a resolution of this important open problem.

One such family of methods are iterative augmentation methods~\cite{dhl-15} using the circuits of the matrix pair $A,B$, which are defined as follows:
\begin{definition}\label{def:circuits}
	Given matrices $A,B$, the set of \textit{circuits} $\mathcal{C}(A,B)$ consists of all $\veg \in \ker(A) \setminus \{ \ve 0 \}$ normalized to coprime integer components for which $B \veg$ is support-minimal over $\{B \vex \mid \vex \in \ker(A) \setminus \{\ve0 \} \}$. 
\end{definition}
The set $\mathcal{C}(A,B)$ is, for instance, the set of all \emph{potential} edge directions, arising from any polyhedron having $A$ and $B$ as their constraint matrices over the varying choices of the right-hand sides $\veb$ and $\ved$. This set contains the set of set of \emph{actual} edge directions appearing on $P$ with $\veb$ and $\ved$ fixed as a subset. To be precise, by an \emph{edge direction}, we mean any (normalized) vector in a one-dimensional subspace spanned by the set of optimal points with respect to some cost vector.
This means that considering all circuits in each iteration gives a potentially larger improvement than by only 
considering the edge directions, as is the case in the Simplex method.

The set of circuits $\mathcal{C(P)}$ of the polyhedron $P = \{ \vex \in \R^n \mid A \vex = \veb, B \vex \leq \ved \}$ satisfies $\mathcal{C(P)}=\mathcal{C}(A,B)$. A generic iterative augmentation method for linear programming over $P$ starts from some initial 
feasible 
solution $\vex^0\in P$ and then, in the $i$-th 
iteration with $i=0, 1, \dots$, finds a circuit $\veg^i \in \CC(A,B)$ and a 
step-length $\lambda^i \in \R_+$ such that $\vex^{i+1} = \vex^i + \lambda^i \veg^i$ is 
feasible and $\vecc^T \veg^i < 0$. The specific choice of $\lambda^i$ and $\veg^i$ distinguishes the individual methods.
For example, a \emph{steepest-descent step} is one which minimizes $\vecc^T \veg^i / \|\veg^i\|_1$, where $\|\cdot\|_1$ is the $1$-norm, and then chooses the largest feasible step-length $\lambda^i$ with respect to $\veg^i$. A \emph{deepest-descent step} is one which minimizes $\lambda^i \vecc^T \veg^i$.

This gives rise to the circuit diameter conjecture \cite{bfh-14}, which states that for any $d$-dimensional polyhedron with $f$ facets, the circuit diameter is bounded above by $f-d$; the circuit diameter is the smallest number of feasible circuit steps of maximal length between two points of a polyhedron.
The significance of studying the circuit-based iterative augmentation methods is also highlighted by recent success of Graver bases in the design of integer programming algorithms~\cite{EisenbrandHKKLO:2019}, since a Graver basis is essentially the integer programming analogue of the set of circuits. 

Throughout this paper, we consider polyhedra in the general form $P = \{ \vex \in \R^n \mid A \vex = \veb, B \vex \leq \ved \}$, just as we already did up to this point. 
We assume that $P$ is pointed, i.e., it has a vertex. This is required for some of our problem statements to be well-defined. A check whether $P$ is pointed can be done efficiently through elementary linear algebra.

Let us formally define a deepest-descent step:
\begin{definition}
	Let $P = \{ \vex \in \R^n \mid A \vex = \veb, B \vex \leq \ved \}$, let $\vecc\in \mathbb{R}^n$ and $\vex^0\in P$, and consider the LP $\min\{\vecc^T\vex: \vex\in P\}$. A {\em $\vecc$-deepest-descent step $\vey$ from $\vex^0$} is a vector $\vey=\alpha \cdot \veg$ for some circuit $\veg \in \mathcal{C}(A,B)$ that maximizes the objective function improvement $-\vecc^T(\alpha \veg)$ among all circuits $\veg \in \mathcal{C}(A,B)$ and all $\alpha > 0$ with $\vex^0+\alpha\veg \in P$. 
\end{definition}
When the context is clear, we simply refer to a deepest-descent step $\vey$ (dd-step), dropping information about $\vecc$, $P$, or $\vex^0$. We call the term $c_{\vey} = -\vecc^T\vey$ the {\em deepest-descent improvement} (dd-improvement).
It is known that repeatedly taking deepest-descent steps converges to an optimal solution in $\Oh(n \log(\veb, \vecc, \ved))$ iterations~\cite{dhl-15}.
A \emph{$k$-approximate dd-step} $\vez$ is a circuit step whose improvement is at least $1/k$ of the improvement of a dd-step, as measured by the objective value $c_{\vey}$ of a dd-step versus $\vecc^T\vez$.
It is known~\cite{AltmanovaKK:2018,dks-19} that iteratively augmenting $k$-approximate dd-steps takes at most $k$-times more iterations to converge to an optimum.
Thus, we are interested in exact and approximate computations of a deepest-descent step. We formally denote this search as follows.

\defProblemFind{Deepest-Descent Step Problem (dd-SP)}
{$\vecc\in \mathbb{R}^n$, polyhedron $P\subset\mathbb{R}^n$, $\vex^0\in P$}
{$\vecc$-deepest-descent circuit step $\vey$ in $P$ from $\vex^0$.}
The natural question leading to our results is then: How hard is it to compute a dd-step?

\subsection{Our Contribution}
Our first positive result with respect to this question pertains to the efficient approximability of \textsc{dd-SP}:
\begin{theorem}\label{thm:deepest-descent-approximation}
	\textsc{dd-SP} can be approximated within a factor of $n$ in polynomial time.
\end{theorem}
This follows by an averaging argument on well-behaved decompositions of the difference of two solutions to an LP as a set of (scaled) circuits.

The obvious follow-up question is whether an $n$-approximation can be significantly improved.
We answer this negatively, even for a fairly restricted family of LPs:
\begin{theorem}\label{thm:deepest-descent-unique}
	Even for LPs over $0/1$-polytopes defined by a totally unimodular matrix and with unique optima,
	\textsc{dd-SP} cannot be approximated 
	within $O(n^{1-\epsilon})$ for any $\epsilon > 0$ in polynomial time, unless $\PP 
	= \NP$.
\end{theorem}
In particular, this demonstrates that to obtain a better approximation ratio or even polynomial tractability, one would need to consider an even more restricted family of LPs.

Further, we turn to the complexity of computing dd-steps exactly. De Loera et al.~\cite{dks-19} have recently shown that \textsc{dd-SP} is \NPh.
However, a closer look at their construction reveals that they in fact show hardness of detecting whether it is possible to get to some optimum in one circuit step from a given initial point $\vex^0$. We call this problem OCNP:

\defProblemDecide{Optimal Circuit-Neighbor Problem (OCNP)}
{$\vecc\in \mathbb{R}^n$, polyhedron $P\subset\mathbb{R}^n$, $\vex^0\in P$ }
{Is there an optimum $\vex^*$ with respect to $\min\{\vecc^T\vex: \vex\in P\}$ such that $\vex^*-\vex^0$ is a circuit direction?}
Somewhat surprisingly, we show:
\begin{theorem}\label{thm:circuit-neighbor-efficiency}
	OCNP is solvable in polynomial time for LPs with a unique optimum.
\end{theorem}
The standard trick of slightly perturbing the objective $\vecc$ of an LP makes some optimum unique, and the set of objectives with non-unique optima has volume $0$, so in a sense OCNP is easy ``almost always.'' This is contrasted by De Loera et al.~\cite{dks-19} showing the $\mathsf{NP}$-hardness of general OCNP.

This raises the following question: What is the complexity of \textsc{dd-SP} for LPs with a unique optimum, given that OCNP is easy?
Despite the encouraging polynomial-time solvability of OCNP for this special case,
we obtain as a byproduct of  Theorem \ref{thm:deepest-descent-unique} that, unlike OCNP, \textsc{dd-SP} remains hard, even for the same, restricted family of LPs:
\begin{theorem}\label{thm:deepest-descent-hardness}
	\textsc{dd-SP} is 
	\NPh, even for LPs over $0/1$-polytopes defined by a totally unimodular matrix and with unique optima.
\end{theorem}

%



\subsection{Connections to Previous Work}\label{sec:related}
There are two papers in the literature with an especially strong connection to ours. 
We detail this connection separately, and discuss other related work hereafter. 

Firstly, and most importantly, an inspiration for this note is the recent paper of De Loera et al.~\cite{dks-19}. Our polynomial-time algorithm for OCNP (Theorem~\ref{thm:circuit-neighbor-efficiency}) stands in contrast to the results of De Loera et al.~\cite{dks-19},
where it is shown that finding optimal circuit-neighbors is \NPh in general.
Hence, the hardness of OCNP hinges on the existence of multiple optima.
At this point, a flawed line of reasoning might become appealing:

The reduction in \cite{dks-19} comes from the directed Hamiltonian path problem.
By introducing a negligible probability for one-sided error through the Isolation Lemma \cite{isolation},
we may assume that the reduction source instance $D = (N,F,s,t)$ on $n=|N|$ nodes, 
has a unique solution---that is, a unique Hamiltonian path from $s$ to $t$.
It is tempting to apply the reduction of \cite{dks-19},  
and use the above algorithm for OCNP to solve the produced instance.
This is an optimization problem on the matching polytope $P_M(H)$ of some undirected bipartite graph $H$ on $2n+1$ vertices.
We then have also solved the original instance of the (unique) Hamiltonian path problem in polynomial time.
This fails, however, since the optima of the instance of OCNP are not in one-to-one correspondence with
Hamiltonian paths in the input instance.
Namely, the set of optima in the instance of OCNP is the set of matchings of size $n-1$ in a graph $H'$ obtained from $H$ through the deletion of some edges.
In particular, this set is not necessarily a singleton if the original graph $D$ had a unique Hamiltonian path.

To save this approach, one might apply a perturbation to the cost vector of the produced LP on $P_M(H)$,
to ensure uniqueness of solutions nonetheless (as remarked, uniqueness holds with probability 1). 
This perturbation, however, would have to retain precisely all optimal circuit neighbors, and not one of the other optima. Producing this perturbation would therefore require us to have at hand an optimal circuit neighbor in the first place.

To avoid confusion, we stress that ``uniqueness'' refers not to the solutions of OCNP itself,
but to the LP that constitutes part of the input of OCNP (which implies uniqueness of the solution for OCNP). 
In other words, there might be a unique optimal circuit neighbor, while the LP has several optima.

Also note that while~\cite{dks-19} discusses approximability, it does \emph{not} concern \textsc{dd-SP} but a different problem: deciding what is the shortest path between two vertices of a polytope, either using the edges of the $1$-skeleton, or using circuit steps.
It is not clear to us whether any inapproximability of \textsc{dd-SP} follows from their construction.

Secondly, we make use of \cite{bv-18} for our positive results on OCNP and the $n$-approximability of \textsc{dd-SP}. Most importantly, the set of circuit directions appear as a subset of the extreme rays of a polyhedral cone constructed from the original input~\cite[Theorem 3]{bv-18}. Recall that extreme rays are those not in the conic hull of any other rays in the object at hand. 

\begin{proposition}\label{prop:circuits_as_rays}
Let $P = \{ \vex \in \R^n \mid A \vex = \veb, B \vex \leq \ved \}$ be a pointed polyhedron. The pointed cone
\begin{align*}
C_{A,B} = \{(\vex, \vey^+, \vey^-) \in \R^{n + 2m_B} \mid A\vex = \ve0, \ B\vex = \vey^+ - \vey^-, \ \vey^+, \vey^- \geq \ve0\}
\end{align*}
is generated by the set of extreme rays $S \cup T'$, where:
\begin{enumerate}
\item The set $S := \{ (\veg, \vey^+, \vey^-) \mid \veg \in \mathcal{C}(A, B), \ y_i^+ = \max\{(B\veg)_i,0\}, \  y_i^- = \max\{-(B\veg)_i,0\}\}$ gives the circuits of $P$.
\item The set $T' \subseteq T := \{ (\ve0, \vey^+, \vey^-) \mid y^+_i = y^-_i = 1 \text{ for some $i \leq m_B$}, \  y^+_j = y^-_j = 0 \text{ for $j \neq i$}\}$ has size at most $m_B$. 
\end{enumerate}
\end{proposition}
Informally, all circuits of $P$ can be found as extreme rays of $C_{A,B}$: a projection of a vector in the set $S$ onto its first $n$ components gives the corresponding circuit $\veg$. The `non-circuits' in the set $T$ are trivial to identify, and the corresponding projection just returns $\ve0$. Note further that the length of a bit encoding of $C_{A,B}$ is (in the order of) at most twice the bit encoding length of $P$. This implies that one can efficiently optimize linear objective functions over the set of (one-normed) circuits. Further, this allows the efficient computation of a conformal sum.

\subsection{Related Work}
Apart from the directly related papers mentioned in the previous subsection, there is vast literature revolving around pivoting rules for circuit augmentation algorithms, and circuits of linear programs in general.
Without any pretense of being comprehensive, let us point to a couple of seminal works (below) and refer to~\cite{bv-18} with respect to circuits, and to~\cite{dks-19} for circuit augmentation and the references therein for a more extensive treatment.

The idea of performing augmenting steps in the direction of circuits instead of only edges during an execution of the simplex algorithm goes back at least to Bland's thesis \cite{bland-thesis} and is explored in detail in \cite{dhl-15} and implemented, for example, in \cite{bv-20}. The notion of a circuit itself in turn was conceived only slightly before that by Rockafellar~\cite{r-69}, and quite fruitfully~\cite{h-99,h-02,h-03} adapted to the integral case by Graver~\cite{g-75}.
 

\subsection{Outline}\label{sec:outline}
Our main contribution is a proof of the inapproximability of the computation of a dd-step within a factor of $O(n^{1-\epsilon})$, even when restricted to special classes of polyhedra. We begin by connecting to and generalizing previous results in the literature, in Sections \ref{sec:optneighbor} and \ref{sec:napprox}. In Section \ref{sec:inapprox}, we prove our main result. In Section \ref{sec:open}, we conclude with some open questions.


\section{Efficiency of OCNP for LPs with unique optima}\label{sec:optneighbor}
We begin by discussing the OCNP problem.
De Loera et al.~\cite{dks-19} showed that OCNP is \NPh, and this implies that computing an optimal dd-step is \NPh.
(We call an optimization problem \NPh if a corresponding decision version---is it possible to meet or exceed a given objective function value?---is \NPh.) Recall the discussion in Section \ref{sec:related}.

The proof in~\cite{dks-19} is based on the underlying LP having multiple optima. While showing the claim under this assumption clearly is sufficient, note that for a given polyhedron $P$, the set $C_{\text{multi}}\subset \mathbb{R}^n$ of $\vecc$ for which there exist multiple optima has volume $0$ in $\mathbb{R}^n$. Informally, it is enough to slightly perturb the objective function to create a unique optimum.
We now show that this hardness does not hold if the underlying LP has a unique optimum. 
\begin{lemma}\label{lem:circuit-neighbor-efficiency} 
Let $\min\{\vecc^T\vex: \vex\in P\}$ be an LP over a polyhedron $P$ with a unique optimum $\vex^*$ (that may not be known), and let $\vex^0\in P$. In polynomial time, it can be verified whether $\vex^*-\vex^0$ is a circuit direction.
\end{lemma}

\begin{proof}
Generally, LPs are solvable in polynomial time. As the LP at hand has a unique optimum $\vex^*$, this $\vex^*$ can be found in polynomial time. Let $\ved = \vex^*-\vex^0$. If $\ved = \ve0$, there is nothing to prove: $\vex^0$ itself already is optimal  and we were able to verify so efficiently. Thus $\ved \neq \ve0$ in the following.

Recall that the circuit directions of a polyhedron $P$ appear as a subset $S$ of the extreme rays of a polyhedral cone $C_{A,B}$, as in Proposition \ref{prop:circuits_as_rays} \cite{bv-18}.  We construct $\ved_S:= (\ved, \vey^+, \vey^-)$, where $\vey_i^{\pm} = \max\{\pm(B\ved)_i,0\}$ as in the definition of $S$. The construction of $\ved_S$ is efficient: $\ved$ is copied over and $\vey^{\pm}$ is derived from a matrix-vector product on the original input and component-wise comparisons. 

 Note that $\ved_S\in C_{A,B}$, as $\vex^{\ast},\vex^0 \in P$, and that $\ved_S \notin T$ (as $\ved \neq \ve0$). Thus, if $\ved_S$ is an extreme ray of $C_{A,B}$, it can only be a member of $S$, which would imply that $\ved$ is a circuit. A check whether a given $\ved_S\in C_{A,B}$ is an extreme ray is possible in polynomial time: first, identify the set of active constraints of $\ved_S$ with respect to $C_{A,B}$, i.e., check which constraints in the formulation of $C_{A,B}$ given in Proposition \ref{prop:circuits_as_rays} are satisfied with equality, and construct the associated row submatrix of all active constraints; then perform a rank check for this submatrix -- if its rank is precisely $(n+2m_B) -1$, then $\ved_S$ lies in a one-dimensional face of $C_{A,B}$, i.e., in an extreme ray. These steps are possible in polynomial time because the bit encoding length of $C_{A,B}$ is at most twice the bit encoding length of $P$.


Summing up, $\vex^{\ast}$ can be found efficiently, $\ved_S$ can be constructed efficiently, and $\ved_S$ is an extreme ray of  $C_{A,B}$ if and only if $\vex^{\ast}-\vex^0$ is a circuit direction, and the required check is efficient, too. This proves the claim.
\end{proof}
As an immediate consequence, we obtain the following theorem.
\begin{reptheorem}{thm:circuit-neighbor-efficiency}
OCNP is solvable in polynomial time for LPs with a unique optimum.
\end{reptheorem}
Because the set of objective functions for which there exist multiple optima for a given polyhedron has volume $0$, Theorem \ref{thm:circuit-neighbor-efficiency} tells us that OCNP ``almost always'' can be decided efficiently. 

\section{$n$-Approximability of \textsc{dd-SP}}\label{sec:napprox}
Next, we show that an efficient approximation of \textsc{dd-SP} with an error equivalent to the dimension of the underlying polyhedron is possible.

\begin{lemma}\label{lem:deepest-descent-approximation}
Let $P = \{ \vex \in \R^n \mid A \vex = \veb, B \vex \leq \ved \}$, let $\vecc\in \mathbb{R}^n$ and $\vex^0\in P$, and consider the LP $\min\{\vecc^T\vex: \vex\in P\}$.
Then an $\left(n{-}\rank(A)\right)$-approximation of a $\vecc$-deepest-descent circuit step in $P$ from $\vex^0$ can be computed in polynomial time.
\end{lemma}


\begin{proof}
Let $\vey$ be a $\vecc$-deepest-descent step $\vey$ in $P$ from $\vex^0$ and let $\vex^{\ast}$ be an optimum of $\min\{\vecc^T\vex: \vex\in P\}$. LPs generally are polynomial-time solvable, so an optimal $\vex^{\ast}$ can be computed efficiently. 

The vector $\vex^{\ast}-\vex^0$ can be written as a so-called {\em conformal sum} $\vex^{\ast}-\vex^0=\sum_{i=1}^{n'} \alpha_i \veg_i$, where $n'\leq n$, $\alpha_i>0$ and $\veg_i$ is a circuit of $P$ for all $i\leq n'$, and all the circuits $\veg_i$ are sign-compatible with each other (and with $\vex^{\ast}-\vex^0$) \cite{g-75,o-10}. In fact, $n'\leq n{-}\rank(A)$, as a lifting of $\vex^{\ast}-\vex^0$ lies in a cone of dimension at most $n{-}\rank(A)$ whose extreme rays correspond to sign-compatible circuits to it~\cite[Theorem 6 and paragraphs after]{bv-18}. Such a conformal sum can be computed in polynomial time, see, e.g., Algorithm 4 in \cite{bv-18}. 

Next, note that $\vecc^T\vey \geq \vecc^T(\vex^{\ast}-\vex^0)=\sum_{i=1}^{n'} \vecc^T(\alpha_i \veg_i)$. (Recall that $\vecc^T\vey$ is negative, as LP is a minimization problem.) Thus, at least one of the $\alpha_i \veg_i$ satisfies $\vecc^T(\alpha_i \veg_i) \leq  \frac{1}{{n'}}\vecc^T(\vex^{\ast}-\vex^0) \leq \frac{1}{{n'}}\vecc^T\vey$. For a given conformal sum, it is efficient to find an $\alpha_i \veg_i$ with smallest value $\vecc^T(\alpha_i \veg_i)$.

 By sign-compatibility of the $\veg_i$, for any index set $I\subset \{1,\dots,n'\}$, $\vex^0+\sum_{i\in I} \alpha_i \veg_i$ lies in $P$. In particular, this holds for $|I|=1$: each of the $\veg_i$ allows for a (maximal-length) circuit step $\beta_i \veg_i$ from $\vex^0$ that stays in $P$, and where $\beta_i \geq \alpha_i$. Note $\vecc^T(\beta_i \veg_i)\leq \vecc^T(\alpha_i \veg_i)$

For a given $\veg_i$, it is efficient to compute the maximal $\beta_i$ such that $\vex^0+\beta_i\veg_i \in P$: each facet of the polyhedron provides an upper bound on $\beta_i$ and one picks the smallest from them. Thus a $\beta_i\veg_i$ with $\vecc^T(\beta_i \veg_i)\leq \frac{1}{{n'}}\vecc^T\vey$ can be computed in polynomial time. This proves the claim.
\end{proof}

As an immediate consequence, we obtain the following corollary.

\begin{reptheorem}{thm:deepest-descent-approximation}
\textsc{dd-SP} can be efficiently approximated within a factor of $n$.
\end{reptheorem}

\section{$O(n^{1-\epsilon})$-Inapproximability of \textsc{dd-SP}}\label{sec:inapprox}
The efficiency of OCNP for LPs with a unique optimum (Section \ref{sec:optneighbor}) is one of the reasons for our interest in a proof for the inapproximability (and implied NP-hardness) of \textsc{dd-SP} that does not rely on this restriction. In Section \ref{sec:napprox}, we saw that there is an efficient $n$-approximation. In this section, we show that this is essentially the best one can expect.

We will provide a proof for the claimed inapproximability of \textsc{dd-SP} that holds even when restricted to special classes of polyhedra. We call a polyhedron  $P = \{ \vex \in \R^n \mid A \vex = \veb, B \vex \leq \ved \}$ with totally unimodular constraint matrices $A$ and $B$ and integral right-hand sides a TU-polyhedron.


To this end, we will perform a reduction from the following problem.

\defProblemFind{Directed Weighted Longest Cycle Problem (DWLCP)}
{Directed graph $G=(V,E,c)$ with arc costs $\vecc\in \mathbb{Q}^{|E|}$ }
{Directed cycle of maximum cost}


DWLCP is a generalization of the Directed (Unweighted) Longest Cycle Problem (DLCP), where the number of arcs of a cycle is counted, i.e., $c_{ij}=1$ for all $(i,j)\in E$. Note that $|V|$ is the largest possible cost of a simple cycle for any instance of DLCP. For a graph $G=(V,E,c)$, DLCP cannot be polynomial-time approximated within $|V|^{1-\epsilon}$ for any $\epsilon > 0$, unless $P = NP$ \cite{bhk-04}. This hardness transfers immediately to DWLCP: the cost $c_{\vex}=\vecc^T\vex$ of a longest cycle $\vex$ cannot be polynomial-time approximated within $|c_\vex|^{1-\epsilon}$  for any $\epsilon > 0$.

Through a reduction from DWLCP, we will prove that a dd-step $\vey$ also cannot be polynomial-time approximated within 
$|c_{\vey}|^{1-\epsilon}$ for any $\epsilon > 0$. In our construction, we will guarantee that the underlying LP has a unique solution (and, even stronger, that this fact is known), which allows us to obtain inapproximability and hardness even for such LPs. To this end, we begin with the polynomial construction of an instance of DWLCP from DWLP where all cycles have different costs while retaining the original ``hierarchy'' of costs. We denote the length of a bit encoding of a weighted graph $G$ as $\mathcal{I}_G$. 

\begin{lemma}\label{lem:unique-cost-cycles}
Let $G=(V,E)$ be a directed graph. It is possible to construct a set of arc 
costs $\vecc\in \mathbb{Q}^{|E|}$ in polynomial time such that: all cycles in 
$G'=(V,E,\vecc)$ have different cost, the cost of a cycle is at least its number of arcs, and the cost of a cycle exceeds the number of 
arcs by strictly less than one. Further, the bit encoding length of $G'$ is 
polynomial in the bit encoding length of $G$.
\end{lemma} 

\begin{proof}
Let $G=(V,E)$ be an unweighted directed graph. Let $n=|V|$ and $m=|E|$. First, 
we endow $G$ with weights to obtain the weighted graph $G'=(V,E,\vecc')$, where $c'_{ij}=1$ for all 
$(i,j)\in E$. In this graph, the cost of a cycle is measured through the number 
of arcs. Cycles with the same number of arcs have the same cost. To simplify notation, we will refer to a cycle interchangeably either as a subset of $E$ or as a $0/1$-vector $\vex$ with components $1$ precisely for the arcs on the cycle (a unit flow along the cycle); in particular, for two cycles $C_1, C_2 \subseteq E$ represented by vectors $\vex_1, \vex_2$, by $\vex_1 \setminus \vex_2$ we mean the arc set $C_1 \setminus C_2$. Note that 
$\mathcal{I}_{G'}$ is polynomial in $\mathcal{I}_G$: for each arc, only a 
(constant-size/single-bit) encoding of the number $1$ is needed.

We will prove the claim through a simple perturbation on $\vecc'$ to resolve any ties between cycles. The new, perturbed costs are called $\vecc$. We are going to show that the perturbation is efficient and changes $\mathcal{I}_{G'}$ only polynomially.

Let $\vecc=\vecc' + \vedelta$, where $\vedelta = (\delta_1,\dots,\delta_m)^T$ and $\delta_i = 2^{-i}$. Informally $\delta_1 = \frac{1}{2}$, $\delta_2 = \frac{1}{4}$, $\delta_3 = \frac{1}{8}$, and so on. Each $\delta_i$ can be encoded in at most $m+1$ bits, due to being the inverses of powers of $2$. Thus, each $c_i=c'_i+\delta_i$ can be encoded in at most $m+2$ digits and $\mathcal{I}_{G'} \leq (m+2) \mathcal{I}_{G}$. As $\mathcal{I}_G\geq m$, the change in encoding length is polynomial. Further, $\vecc$ can be constructed in polynomial time.

It remains to prove that all cycles in $G'$ are of different cost with respect to $\vecc$ and that the cost of cycles has increased by less than one. The latter is immediately clear from $\sum_{i=1}^m \delta_i <1$. Note that the number of arcs of a cycle is $\vecc'^T\vex$ and the cost with respect to $\vecc$ is $\vecc^T\vex$. Let $\vex_1$, $\vex_2$ be two cycles and assume $\vecc'^T\vex_1 > \vecc'^T\vex_2$, which in particular implies $\vecc'^T\vex_1 \geq \vecc'^T\vex_2 + 1$. As $\vecc^T\vex_2<\vecc'^T\vex_2+\sum_{i=1}^m \delta_i$ and $\sum_{i=1}^m \delta_i <1$, we have $\vecc^T\vex_1\geq \vecc'^T\vex_1 > \vecc^T\vex_2$.

Finally, consider two cycles $\vex_1\neq\vex_2$ with $\vecc'^T\vex_1 = \vecc'^T\vex_2$.  The cycles have the same number of arcs, so $\vex_1\backslash \vex_2\neq \emptyset$ and $\vex_2 \backslash \vex_1\neq \emptyset$. Let index $k$ be smallest among all arcs in $\vex_1\backslash \vex_2$, and let $l$ be smallest among all arcs used in $\vex_2 \backslash \vex_1$. Without loss of generality, assume $k < l$. Note $\delta_k > \sum_{i=k+1}^m \delta_i$. Thus $\vecc^T\vex_1-\vecc^T\vex_2\geq \delta_k - (\sum_{i=k+1}^m \delta_i) > 0$, i.e., $\vecc^T\vex_1>\vecc^T\vex_2$. This proves the claim.
\end{proof} 

\begin{remark}
It is natural to ask whether it is necessary to introduce numbers of exponential size into $\vecc$ in the Lemma above.
In other words, does every integer vector $\vecc$ which preserves exactly one optimum of $\vecc'$ and does not introduce any new optima have some entry of order $2^n$?
This is open, but observe that if we require something stronger, the answer is ``yes.''

We show that every integer $\vecc$ which breaks all ties between cycles of the same length must have exponential entries.
Clearly the number of cycles of length $n$ can be $\Omega(2^n)$.
Denote $c_{\max} = \|\vecc\|_\infty$.
In order to get a distinct value $\vecc^T \vex$ for every cycle $\vex$ of length $n$, $c_{\max} \in \Omega(2^n)$, as otherwise there are not enough distinct values $\vecc^T \vex$ since clearly $0 \leq \vecc^T \vex \leq n \cdot c_{\max}$.
\end{remark}



We are now ready to prove our main claim.

\begin{reptheorem}{thm:deepest-descent-unique}
Let $P = \{ \vex \in \R^n \mid A \vex = \veb, B \vex \leq \ved \}$ with 
$A,B\in \mathbb{R}^{m\times n}$, let $\vecc\in \mathbb{R}^n$ and  $\vex^0\in 
P$, and consider the LP $\min\{\vecc^T\vex: \vex\in P\}$. A deepest-descent 
circuit step $\vey$ in $P$ from~$\vex^0$ cannot be approximated 
within $O(n^{1-\epsilon})$ for any $\epsilon > 0$ in polynomial time, unless $\PP 
= \NP$. The hardness holds for LPs with unique optima, over $0/1$-polytopes, TU-polyhedra, or any 
combination thereof.
\end{reptheorem}

\begin{proof}
We will prove the claim through a reduction from the Directed Longest Cycle 
Problem (DLCP), for which it was shown in ~\cite{bhk-04} that no 
$|V|^{1-\epsilon}$-approximation can be computed for any $\epsilon>0$ in 
polynomial time, unless $\PP=\NP$, even in graphs of constant maximum 
out-degree $\Delta^+$. By 
\Cref{lem:unique-cost-cycles}, for a given graph $G=(V,E)$ it is possible to 
efficiently construct a weighted graph $G'=(V,E,\vecc)$ in which all cycles have a 
different cost and their cost lies strictly between the number of arcs of the 
cycle and that number plus one. The graph $G'$ can be used as input for a 
Directed Weighted Longest Cycle Problem (DWLCP) and also has constant maximum 
out-degree $\Delta^+$. If there was an efficient 
$|V|^{1-\epsilon}$-approximation for DWLCP on $G'$, then there would be an 
efficient $|V|^{1-\epsilon}$-approximation for DLCP on $G$. We will show that if 
there exists an algorithm to efficiently $O(n^{1-\epsilon})$-approximate \textsc{dd-SP}, then there exists an efficient 
$|V|^{1-\epsilon}$-approximation for DWLCP, and in turn DLCP -- a contradiction 
unless $P= NP$. Further, the move from $G$ to $G'$ will allow us to show that we 
retain this hardness even for LPs with unique optima.

Let $G=(V,E)$ be a directed graph underlying an instance of DLCP and 
$G'=(V,E,\vecc)$ the corresponding weighted directed graph with perturbed costs 
constructed as in \Cref{lem:unique-cost-cycles}, in turn an instance of DWLCP. 
Next, specify capacities $u_{ij}=1$ for each $(i,j)\in E$ to obtain a network 
$G''=(V,E,\vecc,\veu)$. The costs $c_{ij}$ remain unchanged for all $(i,j) \in E$, 
i.e., they are the same as in $G'$. 

This input can be used to specify a circulation problem. Recall that 
a circulation problem is a special case of a minimum-cost-flow problem and has a 
natural representation as an LP. Using the negative costs $-c_{ij}$ (recall the $c_{ij}$ are positive), we obtain  
\begin{align}
\tag{LP}
\begin{split}\label{LP}
\min \ &  -\vecc^T \vex \\
\text{s.t.} \  & A\vex = \ve0 \\
  & \ve0  \leq\vex \leq \ve 1,
\end{split}
\end{align}
where $A$ is the node-arc incidence matrix of $G'$, and $\ve0$ and $\ve1$ are 
vectors of all-zeros and  
all-ones of appropriate dimensions, respectively. The all-ones vector gives the 
capacity constraints. Let 
$P$ refer to the polyhedron forming the feasible region of~\eqref{LP}. As 
node-arc incidence matrices are totally unimodular, and as the right-hand side 
vectors are the integral $\ve0$ and $\ve1$, $P$ is a $0/1$-polytope in 
$\mathbb{R}^n$, with $n=|E|$. There always exists an optimal vertex to 
an LP on a bounded polytope, so an optimal objective function value for~\eqref{LP} is defined through 
a selection of arcs forming a circulation in $G'$. By the same argument as in 
\Cref{lem:unique-cost-cycles}, any subset of arcs sums up to a different total 
cost. Thus~\eqref{LP} has a unique optimal solution.

Next, consider a trivial feasible flow $\vex^0$ defined by $x^0_{ij}=0$ for 
each $(i,j)\in E$. We are going to show that an efficient approximation of the 
dd-step in $P$ from $\vex^0$ would imply an efficient approximation of DLCP.

Recall that the set of circuits of a node-arc incidence matrix $A$ corresponds 
precisely to the simple {\em undirected} cycles underlying the network; a corresponding vector $\veg \in \{-1,0,1\}$ would have entry $1$ for each directed arc used in the `correct' direction and $-1$ for each directed arc used in the `wrong', opposite direction~\cite{ykk-84,ort-05,bdf-16b}.  The same holds for the 
circuits of $P = \{x\in \mathbb{R}^n\mid A\vex = \ve0, \ve0  \leq\vex \leq \ve 1\}$, as the inequality constraints $\ve0  \leq\vex \leq \ve 1$ are represented through a constraint matrix $B=\binom{I}{-I}$, where $I$ is an identity matrix; recall Definition \ref{def:circuits}. 

Since we have $x^0_{ij}=0$ and $u_{ij}=1$ for each $(i,j)\in E$, the step 
length~$\alpha$ can always be set to $1$ for any valid circuit, i.e., if there 
exists $\alpha>0$ with $\vex^0+\alpha\veg \in P$ for a circuit $\veg$, then 
$\vex^0+\veg \in P$ and $\vex^0+\beta\veg \not\in P$ for any $\beta>1$. Further, any circuit $\veg$ with $\vex^0+\veg \in P$ can only have $0,1$ entries, as $x^0_{ij}=0$ for each $(i,j)\in E$. This means that edges can only be used in the correct direction. Therefore, an exact dd-step~$\vey$ for~\eqref{LP} from $\vex^0$ is in one-to-one 
correspondence to a simple {\em directed} cycle of maximum length (as~\eqref{LP} minimizes 
over negative arc costs). 

Recall that by the hardness result in~\cite{bhk-04},  we may assume that the 
maximum out-degree of $G'$ is some fixed constant $\Delta^+$, so in particular 
$|E|\leq \Delta^+|V|$. Assume we had an algorithm that for a given $\epsilon > 
0$ finds an $(n/\Delta^+)^{1-\epsilon}$-approximate dd-step $\vey_\epsilon$; let $\vey$ refer to an exact dd-step, and let the associated dd-improvements be denoted as $c_{\vey_\epsilon}$ and~$c_{\vey}$, 
respectively. Then we have that 
$\frac{c_{\vey}}{c_{\vey_\epsilon}}\leq 
(n/\Delta^+)^{1-\epsilon}=(|E|/\Delta^+)^{1-\epsilon}\leq |V|^{1-\epsilon}$. We 
know that $c_{\vey}=-\vecc^T\vey=-\vecc^T\veg$ and 
$c_{\vey_\epsilon}=-\vecc^T\vey_\epsilon=-\vecc^T(\alpha\veg_\epsilon)$ for some 
$\alpha\in(0,1]$ and some circuits $\veg$ and $\veg_\epsilon$. By the above, we 
may assume that $\alpha=1$, so $c_{\vey_\epsilon}=-\vecc^T(\veg_\epsilon)$.
Since 
$\frac{-\vecc^T\veg}{-\vecc^T\veg_\epsilon}=\frac{c_{\vey}}{c_{\vey_\epsilon}} 
\leq |V|^{1-\epsilon}$, $\veg_\epsilon$ corresponds to a cycle in $G''$ that 
approximates the longest cycle within a factor of $|V|^{1-\epsilon}$ (since by
construction of the cost vector $\vecc$, a cycle has maximum cost if and only 
if it has maximum length). This would imply a polynomial-time 
$|V|^{1-\epsilon}$-approximation algorithm for general DLCP.


The polytope we used in this construction is a $0/1$-polytope with a TU-matrix, 
and the LP at hand has a unique optimum and this fact is known apriori; see 
above. This shows that the hardness of approximation holds even for LPs adhering 
to all these restrictions. This proves the claim. 
\end{proof}

As an immediate consequence, we obtain \NP-hardness of \textsc{dd-SP} from this inapproximability result. 

\begin{reptheorem}{thm:deepest-descent-hardness}
\textsc{dd-SP} is \NPh, even for LPs over $0/1$-polytopes defined by a TU matrix and with a unique optimum.
 \end{reptheorem}

A direct proof of the \NP-hardness of \textsc{dd-SP} would be possible through 
a reduction from Hamiltonian cycle instead of DWLCP, following a similar line 
of arguments as in the proof of Theorem \ref{thm:deepest-descent-unique}. A 
perturbation of the arc costs would not be necessary, and neither would be the 
careful connection of $|V|$ and $|E|$ through the inapproximability of DWLCP for 
graphs with fixed maximum out-degree. However, to obtain the final part of 
Theorem \ref{thm:deepest-descent-hardness} -- that hardness persists even for 
LPs with unique optima -- one would have to reduce from a variant of Hamiltonian 
cycle where the underlying graph has a {\em unique} circulation with a maximal 
number of arcs and one has the apriori information that there exists such a 
circulation. (This property is what would guarantee the existence of a unique 
optimum in (\ref{LP}), and apriori knowledge thereof.) To the best of the 
authors' knowledge, hardness of this variant has not been studied yet in the 
literature. 


\section{Open Problems}\label{sec:open}
We conclude with two open problems related to our results. First, Theorem~\ref{thm:deepest-descent-approximation} shows how to $n$-approximate \textsc{dd-SP}. However, for the purposes of solving an LP using dd-steps, this is irrelevant, as the first step of the algorithm is to completely solve the LP itself. Is there a combinatorial $n$-approximation of \textsc{dd-SP}, i.e., an algorithm, which does not use the polynomial solvability of an LP as a black-box? Actually, this would yield a new algorithm for linear programming, so to make the question well-posed, we ask whether there is a combinatorial $n$-approximation of \textsc{dd-SP} for some non-trivial class of constraint matrices? Secondly, we have shown strong inapproximability of \textsc{dd-SP}. What are (natural) classes of LP instances for which \textsc{dd-SP} admits, e.g., $\log(n)$- or even $c$-approximation, for some constant $c \in \R_+$? Potential candidate classes include uni-- or bimodular LPs, and more generally, LPs with minors of bounded absolute value. Also, structurally restricted classes of LPs might be of interest. In particular, for $n$-fold LPs, which have a special block-structure, an approximation ratio for \textsc{dd-SP} polynomially depending only on the parameters (that is, block size) would be desirable, and would break below the barrier proved in this paper.

{\small\section*{Acknowledgments} 
\vspace*{-0.2cm}
\noindent S. Borgwardt gratefully acknowledges support of this work through NSF award 2006183 {\em Circuit Walks in Optimization}, Algorithmic Foundations, CCF, Division of Computing and Communication Foundations, and through Simons Collaboration Grant 524210 {\em Polyhedral Theory in Data Analytics} before. C. Brand was supported by project 0016976 of OP VVV call 02\_18\_053. M. Kouteck\'y and A. Feldmann were partially supported by Charles University project UNCE/SCI/004 and by the project 19-27871X of GA \u{C}R.}

\bibliographystyle{elsarticle-num}
\bibliography{literature}

\end{document}